\documentclass[12pt,oneside,reqno]{amsart}
      \usepackage{amssymb}

      \theoremstyle{plain}
      \newtheorem{theorem}{Theorem}[section]
      \newtheorem{lemma}[theorem]{Lemma}
      \newtheorem{proposition}[theorem]{Proposition}
      
      \newtheorem{conjecture}[theorem]{Conjecture}

      \theoremstyle{definition}

      \theoremstyle{remark}

      \newcommand{\R}{\mathbb {Z}^{d}}

\begin{document}

% First we specify the top matter (author, title, etc).
%
% Note: All of the top matter items are optional and can be omitted.
% But you will probably want to specify at least the author and title
% and perhaps an abstract.

   % author information

\title[Almost exponential decay for ballistic RWRE]{Almost exponential decay for the exit probability
from slabs of ballistic RWRE}

   % first author

   \author{Enrique Guerra$^{1}$ and Alejandro F. Ram\'\i rez$^{1,2}$
  }

\thanks{
$^1$ Partially supported by Iniciativa Cient\'\i fica Milenio NC120062}

\thanks{
$^2$ Partially supported  by Fondo Nacional de Desarrollo Cient\'\i fico
y Tecnol\'ogico grant 1141094.}

\email{eaguerra@mat.uc.cl, aramirez@mat.puc.cl}
\address{ Facultad de Matem\'aticas\\
Pontificia Universidad Cat\'olica de Chile\\
Casilla 306-Correo 22, Santiago 6904411, Chile\\
T\'el\'ephone: [56](2)354-5466\\
T\'el\'efax: [56](2)552-5916}

%\authorrunning{E. Guerra\and A. F. Ram\'\i rez}
   % current address, usually not needed because it is the same as the
   % regular address

   % title

   % Note that the short title for running heads goes in square
   % brackets.  This is optional.  The long title goes in curly
   % braces.  In the long title, line breaks are indicated by \\.

   % abstract (optional)
   \begin{abstract}
It is conjectured that in dimensions $d\ge 2$
any random walk  in an i.i.d. uniformly elliptic
random environment (RWRE) which is directionally transient
is ballistic.
The ballisticity conditions for RWRE
 somehow interpolate between directional transience and
ballisticity and have served to quantify the gap which
would need to be proven in order to answer affirmatively
this conjecture.
Two important ballisticity conditions introduced by
Sznitman \cite{Sz02} in 2001 and 2002 are the so called conditions $(T')$
and $(T)$: given a slab of width $L$ orthogonal to $l$, condition $(T')$
in direction $l$
is the requirement that the annealed exit probability
of the walk through the side of the slab in the
half-space $\{x:x\cdot l<0\}$, decays faster than $e^{-CL^\gamma}$ for
all $\gamma\in (0,1)$ and some constant $C>0$, while
condition $(T)$ in direction $l$ is the requirement that
the decay is exponential $e^{-CL}$. It is believed that
$(T')$ implies $(T)$. In this article we  show that
$(T')$ implies at least an {\it almost} (in a sense to be
made precise) exponential decay.
   \end{abstract}

   \date{\today}

% This ends the top matter information.
% We can now tell LaTeX to display the top matter.

   \maketitle

\noindent {\footnotesize

}
% Having displayed the top matter, we now proceed to the body of the
% article.

% The body of the article is divided into sections.
% Each section begins with a \section command.

   \section{Introduction}

The relationship between directional transience and ballisticity for
random walks in random environment is one of the most challenging open
questions within the field of random media. In
the case of random walks in an i.i.d. random environment, several ballisticity
conditions have been introduced which quantify the exit
probability of the random walk through a given side of a slab
as its width $L$ grows, with the objective of understanding the
above relation. Examples of these ballisticity conditions
include Sznitman's $(T')$ and $(T)$ conditions \cite{Sz01,Sz02}.
It is conjectured that
condition $(T')$, which requires a
decay of exiting the slab through its back side
faster than $e^{-CL^\gamma}$, for all $\gamma>0$
and some constant $C>0$, is equivalent
to condition $(T)$, corresponding to exponential decay $e^{-CL}$. In this
article we  prove that
condition $(T')$ implies an {\it almost} exponential decay
of the corresponding exit probabilities.

   Let us introduce the random walk in random environment model.
For $x\in\R$ denote its euclidean norm by $|x|_2$.
Let $V:=\{e\in\mathbb Z^d: |e|_2=1\}$ be the set of canonical vectors. Introduce the set $\mathcal{P}$ whose elements are $2d-$vectors $p(e)_{e\,\in \,\R,\,|e|=1}$ such that

   $$p(e)\ge 0,\,\mbox{for all $e\,\in \, V$},\, \sum_{e\,\in \,\R,\,|e|=1}\,p(e)=1.$$
   We define an environment
$\omega:=\{\omega(x):x \in \R\}$ as an element of $\Omega:=
\mathcal{P}^{\R}$, where for each $x\in\R$,  $\omega(x)=\{\omega(x,e):e\in V\}\in
\mathcal{P}$. Consider a probability measure $\mathbb P$ on $\Omega$ endowed
with its canonical product $\sigma$-algebra, so that an environment
is now a random variable  such that the coordinates $\omega(x)$ are i.i.d. under $\mathbb P$. The random walk in the random environment $\omega$ starting from $x\in\R$ is the canonical Markov Chain $\{X_{n}:n \geq 0\}$ on $(\R)^{\mathbb{N}}$ with \textit{quenched law} $P_{x,\omega}$ starting from $x$,
defined by the transition probabilities for each $e\in\mathbb Z^d$ with $|e|=1$
by

$$P_{x,\omega}(X_{n+1}=X_{n}+e | X_{0},\ldots,X_{n})=\omega(X_{n},e)$$
and

$$P_{x,\omega}(X_{0}=x)=1.$$
The \textit{averaged} or \textit{annealed law}, $P_{x}$, is defined as the semi-direct product measure
   $$P_{x}=\mathbb{P} \times P_{x,\omega}$$
   on $\Omega \times (\R)^{\mathbb{N}}$.
Whenever there is a $\kappa>0$ such that

$$
\inf_{e,x}\omega(x,e)\ge\kappa\qquad\mathbb P-a.s.
$$
we will say that the law $\mathbb P$ of the environment is
{\it uniformly elliptic}.

    For the statement of the result, we need some further definitions.
For each subset $A\subset\mathbb Z^d$ we define the first exit time
of the random walk from $A$ as

$$
T_A:=\inf\{n\ge 0: X_n\notin A\}.
$$
Fix a vector $l\in \mathbb{S}^{d-1}$ and $u\in \mathbb{R}$ then define
the half-spaces $H_{u,l}^-:=\{x\in\mathbb Z^d: x\cdot l<u\}$, $H_{u,l}^+:=\{x\in\mathbb Z^d: x\cdot l>u\}$,

    $$T_{u}^{l}:=T_{H_{u,l}^-}=\inf\{n \geq 0,\,X_{n}\cdot l\geq u\}$$
    and
    $$\tilde{T}_{u}^{l}:=T_{H_{u,l}^+}=\inf\{n \geq 0,\,X_{n}\cdot l\leq u\}.$$
    For $\gamma \in (0,1]$, we say that condition $(T)_{\gamma}|l$ holds with respect to direction $l\in\mathbb{ S}^{d-1}$, if
    $$ \limsup_{L \rightarrow \infty}\,L^{\gamma}\,\log\, P_{0}( \tilde{T}_{-L}^{l'} < T_{L}^{l'} )<0, $$
    for all $l'$ in some neighborhood of $l$.
    Furthermore, we define $(T')|l$ as the requirement that condition $(T)_{\gamma}|l$ is satisfied for all $\gamma\in(0,1)$ and condition $(T)|_{l}$ as
the requirement that $(T)_{1}|l$ is satisfied. 
      In \cite{Sz02}, Sznitman proved that when $d\ge 2$ for every
$\gamma\in (0.5,1)$, $(T)_\gamma|l$ is equivalent to $(T')1l$. 
This equivalence was improved in \cite{DR11} and \cite{DR12}
culminating with the work of Berger, Drewitz and Ram\'\i rez who
in \cite{BDR14}
showed that for any $\gamma\in (0,1)$, condition $(T)_{\gamma}|l$ implies $(T')|l$.
As a matter of fact, in \cite{BDR14}, an effective ballisticity
condition, which requires polynomial decay was introduced.
To define this condition, consider $L,\tilde L>0$ and $l\in\mathbb S^{d-1}$
and the box

$$
B_{l,L,\tilde L}:=R\left((-L,L)\times (-\tilde L,\tilde L)^{d-1}
\right)\cap\mathbb Z^d,
$$
where $R$ is a rotation defined by

\begin{equation}
\label{rot}
R(e_{1})=l.
\end{equation}
Given $M\ge 1$ and $L\ge 2$, we say that the polynomial
condition $(P)_M$ in direction $l$ (also denoted by $(P)_M|l$) is satisfied on a box of size $L$
if there exists and $\tilde L\le 70 L^3$ such that

$$
P_0\left(X_{T_{B_{l,L,\tilde L}}}\cdot l<L\right)\le\frac{1}{L^M}.
$$
Berger, Drewitz and Ram\'\i rez proved in \cite{BDR14} that
there exists a constant $c_0$ such that
whenever $M\ge 15d+5$, the polynomial condition $(P)_M|l$
on a box of size $L\ge c_0$ is equivalent to condition $(T')|l$
(see also Lemma 3.1 of \cite{CR14}).
On the other hand, the following is still open.

\medskip

\begin{conjecture}
\label{conjecture1} Consider a random walk in a uniformly elliptic
random environment in dimension $d\ge 2$ and $l\in\mathbb S^{d-1}$. Then,
condition $(T)|l$ is equivalent to  $(T')|l$.
\end{conjecture}
\medskip

To quantify how far are we presently from proving Conjecture
\ref{conjecture1}, we will introduce now
a family of intermediate conditions between conditions $(T')$
and $(T)$. Let $\gamma(L):[0,\infty)\to [0,1]$, with
$\lim_{L\to\infty}\gamma(L)=1$. Let $l\in\mathbb S^d$.
We say that
condition $(T)_{\gamma(L)}|_l$
is satisfied if

$$
\limsup_{L\to\infty}L^{\gamma(L)}\log  P_{0}( \tilde{T}_{-L}^{l'} < T_{L}^{l'} )<0,
$$
for $l'$ in a neighborhood of $l$.
We will call $\gamma(L)$ the {\it effective parameter} of
 condition $(T)_{\gamma(L)}$.
Note that condition $(T)$ is actually equivalent to
$(T)_{\gamma(L)}$ with an effective parameter given by

\begin{equation}
\label{gammat}
\gamma(L)=1-\frac{C}{\log L},
\end{equation}
for any constant $C\ge 0$.
 In 2002 Sznitman \cite{Sz02} was able to prove that
$(T')$ implies  $(T)_{\gamma(L)}$ with
effective parameter

\begin{equation}
\label{gammasznitman}
\gamma(L)=1-\frac{C}{\log L}\sqrt{\log L},
\end{equation}
for some constant $C>0$.

In this paper, we are able to show that
condition $(T')$ implies condition $(T)_{\gamma(L)}$
with an effective parameter $\gamma(L)$ which
is closer to the effective parameter for condition $(T)$
given by (\ref{gammat}). This is the first result
since the introduction of condition $(T')$ by Sznitman
in 2002, which would give an indication that
Conjecture \ref{conjecture1} is true.
To state it,
let us introduce some notation.
Throughout, for each $n\ge 1$, we will
use the standard notation
%Define recursively on $n$ the sequence of functions,
%    $$g_{1}(x):=\log_{8}(x)\, , \mbox{for $x\geq 1$.}$$
%    $$g_{n+1}(x):=g_{n}\circ g_{1}(x),\,  \mbox{for $g_{n}(x) \geq 1$, $n \geq 1$.} $$
%We will also use the more suggestive notation for each $n\ge 1$
%and $x$ in the domain of $g_n(x)$,

$$
\overbrace{\log\circ\cdots\circ\log}^n x,
$$
for the composition of the logarithm function $n$ times
with itself, for all $x$ in its domain.
where the $n$ superscript means that the composition is performed $n$ times.

\bigskip

   \begin{theorem}
\label{theorem1} Let $d\ge 2$, $l\in S^{d-1}$ and $M\ge 15d+5$.
Assume that  condition $(P)_M|l$ is satisfied
on a box of size $L\ge c_0$.
Then there exists a constant $C>0$ and a function $n(L): [0,\infty)\to\mathbb N$
satisfying $\lim_{L\to\infty}n(L)=\infty$,
such that condition $(T)_{\gamma(L)}|_l$ is
satisfied with an effective parameter $\gamma(L)$ given by

$$
\gamma(L)=1-\frac{C}{\log L}\overbrace{\log\circ\cdots\circ\log}^{n(L)}
L.
$$
   \end{theorem}

\bigskip

Let us remark that a priori, even if $n(L)\to\infty$ as
$L\to\infty$, it might happen that the
composition of the logarithm $n(L)$ time is bounded.
Nevertheless, in the case of Theorem \ref{theorem1}, it turns out that

$$
\lim_{L\to\infty}\overbrace{\log\circ\cdots\circ\log}^{n(L)}L=\infty.
$$
Theorem \ref{theorem1} will be proven in the next section, but some remarks are in order.
    The strategy  followed in the proof, roughly speaking, is to use improve the renormalization procedure used by Sznitman
in \cite{Sz02}, to prove $(T)_{\gamma(L)}$ with
$\gamma(L)$ given by (\ref{gammasznitman}) , through the so called
effective criterion. Essentially, our modification
of such a renormalization scheme, is to work with a sequence of
boxes    growing much faster than in Sznitman's approach.
The use of this new sequence of scales, produces
at some points important difficulties in the
proof which have to be properly handled.

   \section{Proof of Theorem \ref{theorem1}}
Throughout the rest of this section, we prove Theorem \ref{theorem1}.
Firstly, in subsection \ref{prelimin}, we will introduce the
basic notation which will be needed to implement the renormalization
scheme, and we will recall a basic result of Sznitman which
provides a bound for quantities involving the exit probability
through the unlikely side of boxes which are one-dimensional in spirit.
In the second subsection, we will introduce a growth condition
which will limit the maximal way in which the scales on the
renormalization scheme can grow, while still giving a useful
recurrence. In the third subsection we will choose an
adequate sequence of scales satisfying the condition of subsection
\ref{growth-condition}, and for which one can make computations. Finally,
in subsection \ref{effective-theorem}, Theorem \ref{theorem1} will be proven
using the scales constructed in subsection \ref{adequate} through
the use of the effective criterion.

\subsection{Preliminaries and notation}
\label{prelimin}
The proof of Theorem \ref{theorem1} will follow the
renormalization method used by Sznitman to prove Proposition 2.3
of \cite{Sz02}. The idea is to use a renormalization
procedure which somehow mimics a one-dimensional computation,
where one go from one scale to the next (larger) one
through one-dimensional formulas where the exit
probabilities of the random walk through slabs at the
smaller scales are involved.

 Following  Sznitman we introduce boxes transversal to direction $l$, which are specified in terms of $\mathcal{B}=(R,L, L', \tilde{L})$, where $L, L', \tilde{L}$ are positive numbers and $R$ is the rotation defined in (\ref{rot}).
    The box attached to $\mathcal{B}$, is
    $$B:=R((-L,L')\times(-\tilde{L},\tilde{L})^{d-1}) \cap \R $$
    and the positive part of its boundary is defined as
    $$ \partial_+ B:= \partial B \cap \{x \in \R, \, x \cdot l\geq L', \, |R(e_{i}) \cdot x|<\tilde{L},\,i\geq2\}.$$
We can now define the following random variable depending
on a given specification $\mathcal B$, analogous to
the quotient in dimension $d=1$ between the probability to jump
to the left and the probability to jump to the right \cite{SW69,So75},
for $\omega\in\Omega$ as

    $$\rho_{\mathcal{B}}(\omega):=\frac{q_{\mathcal{B}}(\omega)}{p_{\mathcal{B}}(\omega)},$$
    where

    $$q_{\mathcal{B}}(\omega):=P_{0,\omega}(X_{T_{B}}\notin \partial_+ B  )=:1-p_{\mathcal{B}}(\omega).$$
    The first step in the renormalization procedure will be to control
the moments of $\rho_{\mathcal B}$ at the two first scales. For this end, consider positive numbers
    $$3\sqrt{d}<L_{0}<L_{1},\,\,\,3\sqrt{d}<\tilde{L}_{0}< \tilde{L}_{1}$$
    along with the box-specifications

$$\mathcal{ B}_{0}:= (R,L_{0}-1, L_{0}+1, \tilde{L}_{0})$$
and
$$\mathcal{ B}_{1}:= (R,L_{1}-1, L_{1}+1, \tilde{L}_{1}).$$
  It is convenient to introduce now the notation

 $$q_{0}:=q_{\mathcal B_{0}},\,p_{0}:=p_{\mathcal B_{0}},\qquad
q_{1}:=q_{\mathcal B_{1}},\,p_{1}:=p_{\mathcal B_{1}},$$
and

\begin{equation}
\label{citerho0}
\rho_0:=\rho_{\mathcal B_0}, \ \rho_1:=\rho_{\mathcal B_1}.
\end{equation}
Let also

    $$N_{0}:=\frac{L_{1}}{L_{0}}\quad {\rm and}\ \ \tilde{N}_{0}:=\frac{\tilde{L}_{1}}{\tilde{L}_{0}}.$$
We will also need to introduce the constant
\begin{equation}
\nonumber
c_1(d)=c_1:=\sqrt{d}.
\end{equation}
Note that
for each pair of points $x,y\in\mathbb Z^d$, there exists a
nearest neighbor path
joining them which has less than $c_1|x-y|_2$ steps.\\

  Let us now recall the following Proposition of Sznitman \cite{Sz02}.
\bigskip

    \begin{proposition}
\label{sznitman}    There exist $c_{2}(d)>3\sqrt{d}$, $c_{3}(d), c_{4}(d)>1$, such that when $N_{0}\geq 3, L_{0}\geq c_{2}, \tilde{L}_{1} \geq  48 N_{0}\tilde{L}_{0}$, for each $a\in(0,1]$ one has that

%\marginpar{\it Revisar si $\tilde L_0\ge 3\sqrt{d}$}

    \begin{eqnarray}
\nonumber
&\mathbb{E}\left[\rho_{1}^{\frac{a}{2}}\right]\le
c_{3}\left\{
 \kappa^{-10c_{1}L_{1}}\left(c_{4}\tilde{L}_{1}^{d-2}\frac{L_{1}^{3}}{L_{0}^{2}}\tilde{L}_{0} \mathbb{E}[q_{0}]\right)^{\frac{\tilde L_{1}}{12N_{0}\tilde L_{0}}}\right.\\
&\label{one}
+\left.\sum\limits_{0\leq m\leq N_0+1}\left(c_{4}\tilde{L}_{1}^{d-1} \mathbb{E}[\rho_{0}^{a}]
\right)^{\frac{[N_{0}]+m-1}{2}}\right\}.
\end{eqnarray}
    \end{proposition}

\subsection{The maximal growth condition on scales}
\label{growth-condition}

   We next recursively iterate inequality (\ref{one})
at different  scales which will increase as fast as possible,
in the sense that a certain induction
condition should enable us to push forward the recursion.

   We next recursively iterate inequality (\ref{one})
at different  scales which will increase as fast as possible,
in the sense that a certain induction
hypothesis should enable us to push forward the recursion.
   Let

\begin{equation}
\nonumber
v:=8, \,\,\alpha:=240
\end{equation}
   and  introduce two sequences of scales
 $L_{k}, \tilde{L}_{k}\,\,\,k\geq0$, such that

 \begin{equation}
\label{eletilde}
L_{0}\geq c_{2}\,, 3\sqrt{d}<\tilde{L}_{0}\leq L_{0}^{3}
\end{equation}
and for $k\ge 0$

\begin{equation}
\label{nklk}
N_{k}\geq 7,\,L_{k+1}=N_{k}L_{k},\,\tilde{L}_{k+1}=N_{k}^{3}\tilde{L}_{k},
\end{equation}
   as well as box-specifications

$$\mathcal{ B}_{k}:= (R,L_{k}-1, L_{k}+1, \tilde{L}_{k}).$$
Note that

\begin{equation}
\label{elektilde}
\tilde{L}_{k+1}=\left(\frac{L_k}{L_0}\right)^{3}\tilde{L}_{0}.
\end{equation}
Introduce also the notation for the respective attached random variables
$$\rho_k:=\rho_{\mathcal B_k}.$$
Throughout, we will adopt the notation

\begin{equation}
\label{ucero}
u_0:=\frac{3(d-1)}{L_0\log\frac{1}{\kappa}},
\end{equation}
and for $k\ge 1$,

$$
u_k:=\frac{u_{0}}{v^{k}}.
$$
We also let

   $$c_5:=2c_{3}c_{4}.$$

\medskip
\noindent {\bf Condition $(G)$.} {\it
We say that the scales $L_k,N_k, k\ge 0$
satisfy condition $(G)$  if

\begin{equation}
\label{g1}
u_{k}N_{k}\geq \alpha c_{1}\,\mbox{{\rm for} $k \geq 0$,}
\end{equation}
and if

\begin{equation}
\label{g2}
c_{5}N_{k+1}^{3(d-1)}L_{k+1}^{3d-1}\kappa^{u_{k+1}L_{k+1}} \leq 1 \,\mbox{for $k \geq 0$.}
\end{equation}

}

\medskip
Let us now state the following lemma which generalizes Lemma 2.2 of
Sznitman (\cite{Sz02}), for scales satisfying  condition $(G)$.
For completeness we include its proof.

\medskip
   
   \begin{lemma}
\label{induction} Consider scales $L_k, N_k, k\ge 0$,
such that
 condition $(G)$ is satisfied. Then,  whenever  $L_{0}\geq c_{2}$, $3\sqrt{d}\leq \tilde{L_{0}}\leq L_{0}^{3}$,
and   $a_{0}\in(0,1]$, we have that

\begin{equation}
\label{phi0}
\varphi_{0}:= c_{4}\tilde{L}_{1}^{d-1}L_{0}\mathbb{E}[\rho_{0}^{a_{0}}]\leq \kappa^{u_{0}L_{0}}.
\end{equation}
   then for all $k\geq 0$,

\begin{equation}
\label{phik}
\varphi_{k}:= c_{4}\tilde{L}_{k+1}^{d-1}L_{k}\mathbb{E}[\rho_{k}^{a_{k}}]\leq \kappa^{u_{k}L_{k}}.
\end{equation}
   with
   $$a_{k}=a_{0}2^{-k},\,\,u_{k}=u_{0}v^{-k}.$$

   \end{lemma}

   \begin{proof}
As in the proof of Lemma 2.2 of \cite{Sz02}, we can conclude   by Proposition \ref{sznitman} that  if $L_o \geq c_2$ (note that by the choice of $N_k$ in (\ref{nklk}), the
other conditions of Proposition \ref{sznitman} are satisfied)  we have that
for $k\ge 0$,

\begin{equation}
\label{phi2}
\varphi_{k+1}\leq c_{3}c_{4}\tilde{L}_{k+2}^{d-1}L_{k+1}\left\{\kappa^{-10c_{1}L_{k+1}}\varphi_{k}^{\frac{N_{k}^{2}}{12}}+\sum_{0\leq m\leq N_{k}+1}\varphi_{k}^{\frac{[N_{k}]+m-1}{2}}\right\}.
\end{equation}
We will now prove inequality (\ref{phik}) by induction on $k$
using inequality (\ref{phi2}).
Since inequality (\ref{phi0}) is identical to inequality (\ref{phik}) with
$k=0$, the induction hypothesis is satisfied for $k=0$. We assume now that it
is true  for $k>0$, along with  inequality (\ref{g1}) of assumption $(G)$
and conclude that

\begin{equation}
\label{uno}
\kappa^{-10c_{1}L_{k+1}}\varphi_{k}^{\frac{N_{k}^{2}}{24}}\leq \kappa^{-10c_{1}L_{k+1}}\kappa^{N_{k}^{2}\frac{L_{k}u_{k}}{24}}\leq 1.
\end{equation}
   Therefore, using (\ref{uno}) and
the fact that $[N_{k}]-1\geq \frac{N_{k}}{2}$ because $N_{k}\geq 7$ we see that

\begin{eqnarray}
\nonumber
&\varphi_{k+1}\leq c_{3}c_{4}\tilde{L}_{k+2}^{d-1}L_{k+1}\left\{ \varphi_{k}^{\frac{N_{k}^{2}}{24}}+L_{k+1}\varphi_{k}^{\frac{N_{k}}{4}}\right\}\\
\label{dos}
&\leq c_{5}\tilde{L}_{k+2}^{d-1}L_{k+1}^{2}\varphi_{k}^{\frac{N_{k}}{8}}\varphi_{k}^{
\frac{N_{k}}{8}},
\end{eqnarray}
where we recall that $c_5=2c_3c_4$.
  Now, by the induction hypothesis (\ref{phik}) we see that
    $$\varphi_{k}^{\frac{N_{k}}{8}}\leq \kappa^{u_{k+1}L_{k+1}}.$$
  Substituting this into (\ref{dos}), we see that it is
enough now to show that
    $$c_{5}\tilde{L}_{k+2}^{d-1}L_{k+1}^{2}\varphi_{k}^{\frac{N_{k}}{8}}\leq 1.$$
 But this is true, by (\ref{g2}) of condition $(G)$,
the induction hypothesis and the inequality $\tilde L_{k+1}\le L^3_{k+1}$
for $k\ge 0$ which follows by induction starting from (\ref{eletilde}).
Indeed, using these facts,

    $$c_{5}\tilde{L}_{k+2}^{d-1}L_{k+1}^{2}\varphi_{k}^{\frac{N_{k}}{8}}\leq c_{5}N_{k+1}^{3(d-1)}L_{k+1}^{3d-1}\kappa^{u_{k+1}L_{k+1}} \leq 1,$$
which ends the proof.
   
    \end{proof}
\medskip

\subsection{An adequate choice of fast-growing scales}
\label{adequate}

We will now construct a sequence of scales $\{L_k:k\ge 0\}$ which
satisfy condition $(G)$, and for which Lemma \ref{induction}
will eventually imply Theorem \ref{theorem1}. This is not
the fastest possible growing sequence of scales, but somehow
it captures the best possible choice of $\gamma(L)$.

Let  $\{f_k:k\ge 1\}$ be a sequence of functions from $[0,\infty)$
to $[0,\infty)$ defined recursively as

$$f_0(x):=1,$$

$$f_{1}(x):=v^{x}$$
and for $k\ge 1$,
    $$f_{k+1}(x):=f_{k}\circ f_{1}(x).$$
Let now, for $k\ge 0$,

\begin{equation}
\label{enektilde}
N_{k}:=\frac{\alpha c_1}{u_0}\frac{f_{\left[\frac{k+2}{2}\right]}\left(\left[\frac{k+1}{2}\right]\right)}{f_{\left[\frac{k+1}{2}\right]}\left(\left[\frac{k}{2}\right]\right)}.
\end{equation}
    According to display (\ref{nklk}), we have the following formula
valid for $k\ge 0$,
\begin{equation}
\label{elek}
L_{k+1}=f_{\left[\frac{k+2}{2}\right]}\left(\left[\frac{k+1}{2}\right]\right)\left(\frac{\alpha c_1}{u_0}\right)^{k+1}L_0.
\end{equation}

\medskip

    \begin{lemma}
    \label{escale}
  There exists a constant $c_6(d)$ such that when $L_0\ge c_6$, the scales $\{L_k:k\ge 0\}$
and $\{N_k:k\ge 0\}$ defined by (\ref{elek}) and
(\ref{enektilde}) satisfy condition $(G)$.
    \end{lemma}
    \begin{proof}
    We begin proving (\ref{g1}) of condition $(G)$.
  Note that (\ref{g1}) is equivalent to

\begin{equation}
    \label{cg1}
    \frac{f_{\left[\frac{k+2}{2}\right]}\left(\left[\frac{k+1}{2}\right]\right)}{f_{\left[\frac{k+1}{2}\right]}\left(\left[\frac{k}{2}\right]\right)v^{k}}\geq 1\quad {\rm for}\ k\ge 0,
    \end{equation}
    which is obviously true for $k=0,\,1$ and $2$. Therefore it is enough to prove
inequality (\ref{cg1}) for $k\geq3$. For this purpose, we
will first show that for all positive integers $n$, and
$a,\,b\ \in [1,\infty)$,  we have that

\begin{equation}
\label{claim1}
f_n\left(a+b\right)\geq f_n (a)f_n(b).
\end{equation}
    To prove (\ref{claim1}), suppose that

    $$
A:=\{n\in \mathbb N:f_n\left(a+b\right)< f_n (a)f_n(b)\,\, \mbox{for some}\,\
a,b \geq 1\}\neq \varnothing.
$$
  Let $m$ be the smallest element of $A$ and remark that $m$
 is greater than $1$. Also, note that
    $$f_m\left(a+b\right)< f_m (a)f_m(b)$$
    for some $a,b \geq 1$. However, note
that for $a,b\ge 1$ one has that
    $$v^{a+b}\geq v^{a}+v^{b}.$$
 Furthermore,  for each $k\ge 0$, the function
 $f_k(\cdot)$ is increasing. Therefore,

\begin{eqnarray*}
&f_{m-1}(v^{a})f_{m-1}(v^{b})=f_{m}(a)f_m(b)\\
&>f_m(a+b)=f_{m-1}(v^{a+b})\geq f_{m-1}(v^{a}+v^{b}).
\end{eqnarray*}
This contradictions  the minimality of $m$ and hence $A=\varnothing$
which proves (\ref{claim1}).

    Back to (\ref{cg1}), note that

    \begin{eqnarray*}
&    \frac{f_{\left[\frac{k+2}{2}\right]}\left(\left[\frac{k+1}{2}\right]\right)}{f_{\left[\frac{k+1}{2}\right]}\left(\left[\frac{k}{2}\right]\right)v^{k}}\geq
    \frac{f_{\left[\frac{k+2}{2}\right]}\left(\left[\frac{k+1}{2}\right]-1\right)}{f_{\left[\frac{k+1}{2}\right]}\left(\left[\frac{k}{2}\right]\right)}
\frac{f_{\left[\frac{k+2}{2}\right]}\left(1\right)}{v^{k}}\geq
   \frac{f_{\left[\frac{k+2}{2}\right]}\left( 1 \right)}{v^{k}} \geq 1,
    \end{eqnarray*}
    where the first inequality was gotten using (\ref{claim1}),
the second one is a consequence of the inequality

    $$\frac{f_{\left[\frac{k+2}{2}\right]}\left(\left[\frac{k+1}{2}\right]-1\right)}{
f_{\left[\frac{k+1}{2}\right]}\left(\left[\frac{k}{2}\right]\right)}\geq 1,$$
valid for $k\ge 3$, and    which can be proved in a straightforward fashion if we divide the argument according to whether $k$ is even or odd,
    and the last inequality comes from the fact that

\begin{equation}
\label{desi}
f_{[\frac{k+2}{2}]-1}(1)-k \geq 0\qquad{\rm for}\ k\ge 3.
\end{equation}
Now, (\ref{desi})
can be proven noting that it is satisfied for $k=3$,
 the left-hand side of (\ref{desi}) achieves
its minimum value for $k=4$, and is increasing for every $k\ge 3$, from
$2k$ to $2k+1$, and from $2k$ to $2k+2$.      This completes the proof of (\ref{cg1}).

    We now prove  inequality (\ref{g2}) of condition $(G)$.
 We need to show that there exists a constant $c(d,\kappa)$,
 such that whenever $L_{0}\geq c(d,\kappa)$, for all $k\ge 0$
one has that

\begin{equation}
\label{ineqprin}
c_{5}N_{k+1}^{3(d-1)}L_{k+1}^{3d-1}\kappa^{u_{k+1}L_{k+1}} \leq 1.
\end{equation}
    We will first show that there exists $c_7(d,\kappa)=c_7(d)>0$,
 such that whenever $L_{0}\geq c_7$, one has that for $k\ge 0$,

\begin{equation}
\label{enek11}
N_{k+1}^{3(d-1)}\kappa^{\frac{u_{k+1}L_{k+1}}{3}}\leq 1.
\end{equation}
    Now (\ref{enek11}) is equivalent to

\begin{eqnarray}
\nonumber
&3(d-1)\log_v\left(\frac{\alpha c_1}{u_0}\,\frac{f_{\left[\frac{k+3}{2}\right]}\left(\left[\frac{k+2}{2}\right]\right)}{f_{\left[\frac{k+2}{2}\right]}\left(\left[\frac{k+1}{2}\right]\right)} \right)\\
\nonumber
&-
 \frac{L_0 u_0 f_{\left[\frac{k+2}{2}\right]}\left(\left[\frac{k+1}{2}\right]\right)\left(\frac{\alpha c_1} {v u_0}\right)^{k+1}
\log_{v}\left( \frac{1}{\kappa} \right)
}{3}
\leq 0.
\end{eqnarray}
Therefore, (\ref{enek11})  is equivalent to the
bound for $k\ge 0$,

\begin{equation}
\label{enek13}
L_0\geq \frac{\frac{9(d-1)}{u_0}\log_{v}\left(\frac{\alpha c_1}{u_0}\frac{f_{\left[\frac{k+3}{2}\right]}\left(\left[\frac{k+2}{2}\right]\right)}{f_{\left[\frac{k+2}{2}\right]}\left(\left[\frac{k+1}{2}\right]\right)}  \right)}{
 f_{\left[\frac{k+2}{2}\right]}\left(\left[\frac{k+1}{2}\right]\right)\left(\frac{\alpha c_1} {v u_0}\right)^{k+1}
\log_{v}\left( \frac{1}{\kappa} \right) }.
\end{equation}
Let us focus in right-hand side of inequality (\ref{enek13}) .  Note
that it can be split as

\begin{equation}
\label{decompostion}
\frac{
\frac{9(d-1)}{u_0}\log_v\left( \frac{\alpha c_1 }{u_0}\right)}{
f_{\left[\frac{k+2}{2}\right]}\left(\left[\frac{k+1}{2}\right]\right)\left(\frac{\alpha c_1} {v u_0}\right)^{k+1}
\log_{v}\left( \frac{1}{\kappa} \right)
}\,+\,\frac{\frac{9(d-1)}{u_0}\log_v\left( \frac{f_{\left[\frac{k+3}{2}\right]}\left(\left[\frac{k+2}{2}\right]\right)}{f_{\left[\frac{k+2}{2}\right]}\left(\left[\frac{k+1}{2}\right]\right)}\right)}{
 f_{\left[\frac{k+2}{2}\right]}\left(\left[\frac{k+1}{2}\right]\right)\left(\frac{\alpha c_1} {v u_0}\right)^{k+1}
\log_{v}\left( \frac{1}{\kappa} \right)
}.
\end{equation}
Let us now try to find an upper bound for this expression
independent on $u_0$ (or
equivalently, on $L_0$). By the definition of
$u_0$ (c.f. (\ref{ucero})) note that for $k\geq0$ and
$L_0 \geq \frac{3(d-1)}{\log\frac{1}{\kappa}}$ one has that,

$$
\frac{1}{u_0}\frac{1}{\left(\frac{\alpha c_1}{ v u_0}\right)^{k+1}}
 =
\frac{1}{\left(\frac{\alpha c_1}{ v u_0}\right)^{k}}\frac{1}{\left(\frac{\alpha c_1}{v}\right)}\leq
\frac{1}{\left(\frac{\alpha c_1}{ v }\right)^{k+1}}.
$$
Substituting this into (\ref{decompostion}) we see that
it is bounded from above by
\begin{equation}
\label{erase}
\frac{
9(d-1) \log_v\left( \frac{\alpha c_1 }{u_0}\right)}{
f_{\left[\frac{k+2}{2}\right]}\left(\left[\frac{k+1}{2}\right]\right)\left(\frac{\alpha c_1} {v }\right)^{k+1}
\log_{v}\left( \frac{1}{\kappa} \right)
}\,+\,\frac{9(d-1) \log_v\left( \frac{f_{\left[\frac{k+3}{2}\right]}\left(\left[\frac{k+2}{2}\right]\right)}{f_{\left[\frac{k+2}{2}\right]}\left(\left[\frac{k+1}{2}\right]\right)}\right)}{
 f_{\left[\frac{k+2}{2}\right]}\left(\left[\frac{k+1}{2}\right]\right)\left(\frac{\alpha c_1} {v }\right)^{k}
\log_{v}\left( \frac{1}{\kappa} \right)
}.
\end{equation}
Note that only the left-most term of (\ref{erase})
depends on $L_0$. Choose a constant $c_8(d,\kappa)=c_8(d)>1$, such that if
$L_0\geq c_8$

$$
\log_v\left(\frac{\alpha c_1}{u_0}\right)\le L_0\frac{\log_v\left(\frac{1}{\kappa}\right)}{d-1}.
$$
Then, when $L_0\ge c_8$, the left-most term of (\ref{erase})
can be bounded by

\begin{equation}
\label{onethird}
L_0\frac{9v}{\alpha c_1}
\le L_0 \frac{72}{240}\le \frac{L_0}{3}.
\end{equation}
Thus, whenever $L_0\ge c_8$, from
 (\ref{decompostion})
(\ref{erase}) and (\ref{onethird}), we see that
(\ref{enek13}) is satisfied if
\begin{equation}
\label{ineforlo}
L_0 \geq \frac{3}{2}\frac{9(d-1)\log_v\left( \frac{f_{\left[\frac{k+3}{2}\right]}\left(\left[\frac{k+2}{2}\right]\right)}{f_{\left[\frac{k+2}{2}\right]}\left(\left[\frac{k+1}{2}\right]\right)}\right)}{
 f_{\left[\frac{k+2}{2}\right]}\left(\left[\frac{k+1}{2}\right]\right)\left(\frac{\alpha c_1} {v }\right)^{k+1}
\log_{v}\left( \frac{1}{\kappa} \right)
}.
\end{equation}
Therefore, in order to prove (\ref{enek11}) it is enough
to show that the right hand side of inequality (\ref{ineforlo})  is bounded.
To do this, it is enough to prove that
the expression

$$
   \frac{\log_v\left(\frac{f_{\left[\frac{k+3}{2}\right]}\left(\left[\frac{k+2}{2}\right]\right)}{f_{\left[\frac{k+2}{2}\right]}\left(\left[\frac{k+1}{2}\right]\right)} \right)}{ f_{\left[\frac{k+2}{2}\right]}\left(\left[\frac{k+1}{2}\right]\right)}
$$
is bounded. Now,

    \begin{equation}
    \label{eq}
   \frac{\log_v\left(\frac{f_{\left[\frac{k+3}{2}\right]}\left(\left[\frac{k+2}{2}\right]\right)}{f_{\left[\frac{k+2}{2}\right]}\left(\left[\frac{k+1}{2}\right]\right)} \right)}{ f_{\left[\frac{k+2}{2}\right]}\left(\left[\frac{k+1}{2}\right]\right)}
\leq \frac{ \log_v\left(f_{\left[\frac{k+3}{2}\right]}\left(\left[\frac{k+2}{2}\right]\right) \right)}{ f_{\left[\frac{k+2}{2}\right]}\left(\left[\frac{k+1}{2}\right]\right)}.\\
    \end{equation}
Let us now remark that if $k$ is even, then
 $\left[\frac{k+3}{2}\right] =\left[\frac{k+2}{2}\right]$
and
 $\left[\frac{k+1}{2}\right] =\left[\frac{k+2}{2}\right]-1$.
 Therefore, in this case, the right-hand side  of inequality (\ref{eq}) is smaller than

$$
\frac{f_{\left[\frac{k+2}{2} \right]-1}\left(\left[\frac{k+2}{2}\right]\right)}{f_{\left[\frac{k+2}{2} \right]}\left(\left[\frac{k+2}{2}\right]-1\right)}
=
\frac{f_{\left[\frac{k+2}{2} \right]-1}\left(\left[\frac{k+2}{2}\right]\right)}{f_{\left[\frac{k+2}{2} \right]-1}\left(v^{\left[\frac{k+2}{2}\right]-1}\right)}.
$$
But, since for $k$ fixed, the function $f_k(\cdot)$ is  increasing,
and since for $k\ge 0$ we have that

    $$v^{\left[\frac{k+2}{2}\right]-1}\ge \left[\frac{k+2}{2}\right],$$
we see that the right-hand side of inequality (\ref{eq}) is bounded.
Hence, for $k$ even the right-most term of (\ref{enek13}) is bounded by a
 constant $c_{9}(d,\kappa)=c_9(d)>0$.

Suppose now that  $k$ is odd. Then $\left[\frac{k+3}{2}\right] =\left[\frac{k+2}{2}\right]+1$ and
 $\left[\frac{k+1}{2}\right] =\left[\frac{k+2}{2}\right]$.
Therefore, in this case,  the right-hand side of inequality (\ref{eq})
is equal to

$$\frac{f_{\left[\frac{k+2}{2} \right]}\left(\left[\frac{k+2}{2}\right]\right)}{f_{\left[\frac{k+2}{2} \right]}\left(\left[\frac{k+2}{2}\right]\right)}=1,$$
so that there is constant $c_{10}(d,\kappa)=c_{10}(d)>0$ which is an upper bound for
the right-hand side of inequality (\ref{enek13}).
We can hence conclude, taking  $c_7(d)=\max\{c_{9}(d),c_{10}(d)\}$,
that when $L_0 \geq c_7(d)$, then (\ref{enek11}) holds.

    As a second step to prove (\ref{ineqprin}),
we will show that it is possible to find a positive constant
$c_{11}(d,\kappa)=c_{11}(d)$ such that when $L_0\geq c_{11}$ one has that for all $k\ge 0$,
    \begin{equation}
    \label{des2a3}
    L_{k+1}^{3d-1}\,\kappa^{\frac{u_{k+1}L_{k+1}}{3}}\,\leq \, 1.
    \end{equation}
    Inserting the definition (\ref{elek}) that defines $L_k$ into
 this inequality, we see that it is enough to prove that

\begin{equation}
\label{in12}
 (3d-1)\log_v\left( L_{k+1}\right) - \frac{ \log_v \left( \frac{1}{ \kappa} \right) u_0 \left( \frac{\alpha c_1}{u_0 v}\right)^{k+1} f_{\left[\frac{k+2}{2}\right]}\left(\left[\frac{k+1}{2}\right]\right)L_0}{3} \leq 0.
\end{equation}
   Now, to prove (\ref{in12}), we need to show that for all $k\ge 0$,
\begin{equation}
\label{in13}
L_0\geq \frac{\log_v\left(L_{k+1}\right)3(3d-1)}{\log_v\left( \frac{1}{\kappa}\right)u_0 \left(\frac{\alpha c_1}{u_0 v}\right)^{k+1}f_{\left[\frac{k+2}{2}\right]}\left(\left[\frac{k+1}{2}\right]\right)}.
\end{equation}
But the right-hand side of inequality (\ref{in13})
can be written as

\begin{equation}
\nonumber
\frac{3(3d-1)\log_v\left[L_0\left( \frac{\alpha c_1}{u_0 }\right)^{k+1}\right]}{\log_v\left( \frac{1}{\kappa}\right)u_0 \left(\frac{\alpha c_1}{u_0 v}\right)^{k+1}
f_{\left[\frac{k+2}{2}\right]}\left(\left[\frac{k+1}{2}\right]\right)}
+
\frac{3(3d-1)\log_v\left( f_{\left[\frac{k+2}{2}\right]}\left(\left[\frac{k+1}{2}\right]\right)\right)}{f_{\left[\frac{k+2}{2}\right]}\left(\left[\frac{k+1}{2}\right]\right)
}.
\end{equation}
We need to establish a control with respect to $L_0$ in this expression. Only
the first term depends on $L_0$ so we concentrate on the second term.
To this end this term is decreasing with $k$. Therefore,
it is smaller than
\begin{equation}
\nonumber
\frac{3(3d-1)\log_v\left[L_0\left( \frac{\alpha c_1}{u_0 }\right)\right]}{\log_v\left( \frac{1}{\kappa}\right) \left(\frac{\alpha c_1}{ v}\right)}=\frac{3(3d-1)\log_v \left( \frac{L_0^2\alpha c_1\log(\frac{1}{\kappa})}{3(d-1) }\right)}{\log_v\left( \frac{1}{\kappa}\right) \left(\frac{\alpha c_1}{ v}\right)}
\end{equation}
From this last expression, it is clear that we can choose a constant $c_{12}(d, \kappa)=c_{12}(d)>0$ such that whenever $L_0\geq c_{12}(d)$ one
has that

\begin{equation}
\frac{3(3d-1)\log_v\left[L_0\left( \frac{\alpha c_1}{u_0 }\right)^{k+1}\right]}{\log_v\left( \frac{1}{\kappa}\right)u_0 \left(\frac{\alpha c_1}{u_0 v}\right)^{k+1}
f_{\left[\frac{k+2}{2}\right]}\left(\left[\frac{k+1}{2}\right]\right)}\leq \frac{L_0}{3}.
\end{equation}
Therefore, if $L_0\ge c_{12}(d)$ and if
\begin{equation}
\label{ss}
L_0\geq\frac{3}{2}\frac{3(3d-1)\log_v\left( f_{\left[\frac{k+2}{2}\right]}\left(\left[\frac{k+1}{2}\right]\right)\right)}{f_{\left[\frac{k+2}{2}\right]}\left(\left[\frac{k+1}{2}\right]\right)
},
\end{equation}
 we would have (\ref{des2a3}), whenever we could prove that the right hand side of (\ref{ss}) is bounded independently of $k\geq0$.
This can be proven in analogy to the previous computations
made to show that the right-hand side of (\ref{ineforlo})
is bounded. We have thus established the existence of
a constant  $c_{11}(d)$ such that (\ref{des2a3}) is satisfied
whenever $L_0\ge c_{11}(d)$.

    On the other hand it is obvious that there  is a constant  $c_{13}(d)$, such that when $L_0\geq c_{13}(d)$, for $k\ge 0$,

    $$c_5\kappa^{\frac{u_{k+1}L_{k+1}}{3}} \leq 1.$$
   Finally, in order for inequality (\ref{g2}) of condition $(G)$
to be fulfilled, it is enough to take $c_6(d):=\max\{c_7(d),\,c_{11}(d),\,c_{13}(d)\}$.

    \end{proof}

\subsection{The effective criterion implies Theorem \ref{theorem1}}
\label{effective-theorem}
We continue now showing how Lemma \ref{induction} with the appropriate choice
of scales, enables us to use the effective criterion to prove
the decay of Theorem \ref{theorem1}. Let us define
for $x\in\mathbb Z^d$,

$$
|x|_\perp:=\max\{|x\cdot R(e_i)|:2\le i\le d\}.
$$
Also, define for each $x\in\mathbb Z^d$, the canonical translation
on the environments $t_x:\Omega\to\Omega$ as

$$
t_x(\omega)(y):=\omega(x+y)\quad{\rm for}\ y\in\mathbb Z^d.
$$
For the statement of the following proposition and its proof, we will use the shorthand
notation for each $n$,

$$
\log_8^{(n)}(L):=\overbrace{\log_8\circ\cdots\circ\log_8}^n (L).
$$
\medskip

    \begin{proposition}
\label{final-prop}
    There exist $c_{15}(d)>1$, $c_{14}(d)\geq 3\sqrt{d}$ such that whenever
  $L_{0}\geq c_{14}$,
$3\sqrt{d}\leq \tilde{L}_{0}\leq L_{0}^{3}$, and for the box specification $\mathcal{B}_{0}=(R,L_0 - 1, L_0 +1, \tilde{L}_0)$, the condition
    \begin{equation}
    \label{ce1}
    c_{15}\left(log\left(\frac{1}{\kappa}\right)\right)^{3(d-1)}\tilde{L}_{0}^{d-1}L_{0}^{3d-2} \inf_{a \in(0,1]} \mathbb{E}[\rho_{0}^{a}]<1,
    \end{equation}
is satisfied (recall the definition of $\rho_0$
in (\ref{citerho0})), then there exist a constant $c > 0$ and a function $n(L): [0,\infty) \rightarrow \mathbb{N}$, with $n(L) \rightarrow \infty$ as $L\rightarrow\infty$, such that
    \begin{equation}
    \label{result}
     \limsup_{L \rightarrow \infty} \,\, L^{-1}\,\exp \{c  \log_8^{n(L)}L\}\, \log P_0(T_{L}^{l} \leq \tilde{T}_{-L}^{l})<0.
    \end{equation}

    \end{proposition}
    \begin{proof}
Let us choose a sequence of scales $\{L_k:k\ge 0\}$
and $\{\tilde L_k:k\ge 0\}$ according to displays
(\ref{elek}) and (\ref{elektilde}). 
With this choice of scales, as in the proof of Proposition 2.3 of Sznitman \cite{Sz02},
one can see that there are  constants $c_{15}(d)$ and 
 $c_{14}\geq \max\{c_6,\,c_2\}$ such that if $L_0\geq c_{14}$
then
condition (\ref{ce1}) implies  condition
(\ref{phi0}) of Lemma \ref{induction} with $u_0$ 
chosen according to (\ref{ucero}).
By Lemma (\ref{escale}), the chosen scales $\{L_k:k\ge 0\}$
and $\{\tilde L_k:k\ge 0\}$ satisfy condition $(G)$.
Therefore, since (\ref{phi0})  of Lemma (\ref{induction})
is satisfied , we know
that for all $k\ge 0$, inequality (\ref{phik}) is satisfied.
The strategy to prove (\ref{result}) will be
similar to that employed in \cite{Sz02} to prove Proposition 2.3:
we will first choose an appropriate
$k$ so that $L_k$ approximates a fixed scale $L$
tending to $\infty$.
Nevertheless, since here we are working with scales
which are much larger than those used in \cite{Sz02}, we will have to
be much more careful with this argument.

Let $L\ge L_0$. Then, there exists
a unique integer $k=k(L)$ such that
    $$L_{k}\leq L < L_{k+1}.$$
 Note that to prove (\ref{result})
it is enough to show that there
exists a positive constant $c_{16}$ such that
for all $L\ge L_0$ one has that

\begin{equation}
\label{result0}
P_{0}(\tilde{T}_{-L}^{l}<T_{L}^{l}) \leq \frac{1}{c_{16}}\exp\left\{-c_{16}L
\exp\left\{-
\frac{1}{c_{16}}
\log_8^{\left(\left[\frac{k+1}{2}\right]\right)}(L) \right\}\right\}.
\end{equation}
In effect, since clearly $k\to\infty$ as $L\to\infty$,
choosing $n(L)=\left[\frac{k+1}{2}\right]$ we have (\ref{result}).

We will divide the proof of (\ref{result0})  into two cases.

\medskip

\noindent    \emph{Case 1}. Assume that

\begin{equation}
\label{case1}
L \leq \frac{2 \alpha c_{1}}{u_{0}}v^{k}L_{k}.
\end{equation}
 Let

    $$B:=\left\{x \in \R:|x|_\perp \leq \left[\frac{L}{L_{k}}\right]\tilde{L}_{k}, \, x \cdot l\in(-L,L) \right\}.$$
From the inequality
 $\mathbb{E}[q_{k}]\leq\mathbb{ E}[\rho_{k}^{a_{k}}]$, Lemma \ref{induction}
 and Chebyshev inequality, we see that
if

    $$\mathcal{H}:=
\{\omega\in\Omega:\exists x \in B\ {\rm such}\ {\rm that}\ \,\, q_{k}\circ t_{x}(\omega)\geq \kappa^{\frac{1}{2}u_{k}L_{k}}\},$$
then

    $$\mathbb{P}(\mathcal{H}) \leq \kappa^{\frac{1}{2}u_{k}L_{k}}\frac{|B|}{\tilde{L}_{k+1}^{d-1}L_{k}}.$$
Note that
    on $\mathcal{H}^{c}$, by the strong Markov property  one has that

    $$P_{0,\omega}(T_{L}^{l} \leq \tilde{T}_{-L}^{l})\geq (1-\kappa^{\frac{1}{2}u_{k}L_{k}})^{\left[\frac{L}{L_{k}}\right]+1}.$$
  Therefore, since for $x\in[0,1]$ and $n$ natural
one has that $(1-x)^n\le n(1-x)$, for $L$ large enough

\begin{eqnarray}
\nonumber
& P_{0}(\tilde{T}_{-L}^{l}<T_{L}^{l})\leq \left(\frac{|B|}{\tilde{L}_{k+1}^{d-1}L_{k}}+\frac{L}{L_{k}}+1 \right)\kappa^{\frac{1}{2}u_{k}L_{k}}\\
\nonumber
&\le 3 \times 2^d 
\left(\frac{L}{L_k}\right)^d\kappa^{\frac{1}{2}u_{k}L_{k}}\\
\label{leading}
&\le 3 \times 2^d 
 \left(\frac{2\alpha c_1v^{k}}{u_0}\right)^d\kappa^{\frac{1}{4}u_{k}L_{k}}\leq 1,
\end{eqnarray}
where in the third inequality we have used 
our assumption on $L$ (\ref{case1}).
Hence,  we can check that there
is a constant $c_{17}$, such that
for $k\ge 0$,

    \begin{equation}
    \label{equ1}
    P_{0}(\tilde{T}_{-L}^{l}<T_{L}^{l}) \leq
\frac{1}{c_{17}} \exp \left\{- c_{17}\frac{L_k}{v^{k}} \right\}.
    \end{equation}
 Now, again by our assumption  (\ref{case1}), observe that
there is a constant $c_{18}$ such that

    \begin{equation}
    \label{equ2}
    \frac{L_k}{v^{k}}> c_{18}\frac{L}{v^{2k}}.
    \end{equation}
On the other hand, note that
when $L_0\ge\sqrt{\frac{3(d-1)}{\alpha c_1 \log\frac{1}{\kappa}}}$, we
have by the choice scales given in (\ref{elek}), that for $k\ge 1$

 \begin{equation}
\label{fek}
f_{\left[\frac{k+1}{2}\right]}\left(\left[\frac{k}{2}\right]\right)\leq L_k \leq L.
\end{equation}
Repeatedly taking logarithms in (\ref{fek}), we conclude that
for $k\ge 1$
    \begin{equation}
    \label{equ3}
    \frac{k}{4}\leq\left[\frac{k}{2}\right]\leq \log_8^{
\left(\left[\frac{k+1}{2}\right]\right)}(L).
    \end{equation}
  Then, substituting the inequalities
 (\ref{equ2}) and (\ref{equ3}) into (\ref{equ1}),
we see that there exists a positive constants $c_{16}$ 
such that for $L\ge L_0$

    $$P_{0}(\tilde{T}_{-L}^{l}<T_{L}^{l}) \leq \frac{1}{c_{16}}\exp
\left\{-c_{16}L\exp\left\{-\frac{1}{c_{16}}
\log_8^{\left(\left[\frac{k+1}{2}\right]\right)}(L) \right\}\right\}.$$
   Now, (\ref{result}) follows taking $n(L)=\left[\frac{k+1}{2}\right]$.

    \medskip
\noindent    \emph{Case 2}. Let us now assume that

$$
L > \frac{2 \alpha c_{1}}{u_{0}}v^{k}L_{k}.
$$
Let
    $m_{k}$ be the unique integer such that
    $$m_{k}L_{k}\leq L  < (m_{k}+1)L_{k}.$$
 By the definition of $m_k$ we have the inequality
    \begin{equation}
    \label{mk}
    m_{k}\geq \frac{\alpha c_{1}}{u_{0}} v^k.
    \end{equation}
We will now follow an approach similar to the one
employed for {\it Case 1}, but using a sequence of
scales which approximate $L$ with a higher precision than
the $\{L_k\}$ sequence. 
    Let us define

    \begin{eqnarray}
    \label{lkformk}
 S_1^k &:=&m_{k}L_{k},\\
 \nonumber
 \widetilde{S}_1^k&:=&m_k^3 \widetilde{L}_k,\\
 \nonumber
 S_2^k&:=&m_k^2L_k,\\
 \nonumber
 \widetilde{S}_2^k&:=&m_k^6 \widetilde{L}_{k},
    \end{eqnarray}
  along with the
 box-specification $\mathcal{\widetilde{B}}:=(R,S_1^k-1,S_1^k +1,\widetilde{S}_1^k )$ and the random variable $\widehat{\rho}_{k}$ attached to this
    box-specification.
In analogy with the proof of Lemma \ref{induction}, we will prove that
    \begin{equation}
    \label{desfinal}
    (\widetilde{S}_2^k)^{d-1} S_1^k \mathbb{E}[\widehat{\rho}_{k}^{a_{k+1}}]\leq \kappa^{u_{k+1}S_1^k}.
    \end{equation}
    For the time being, assume that this inequality is true.
Let

    $$\widehat{B}=\left\{x \in \R:|x|_{\perp} \leq \left[\frac{L}{S_1^k} \right]
\tilde{S}_1^{k}, \, x \cdot l\in(-L,L) \right\}.$$
    In analogy with the development
of {\it Case 1},  using (\ref{desfinal}) we can arrive to the
following inequality analogous to (\ref{leading}) 
    $$P_{0}[\tilde{T}_{-L}^{l}<T_{L}^{l}]\leq \left(\frac{|\widehat{B}|}{(\widetilde{S}_2^k)^{d-1} S_1^k }+\frac{L}{S_1^k}+1 \right)\kappa^{\frac{1}{2}u_{k+1} S_1^k}.$$
  From here we conclude that there is a constant $c_{19}$ such that
for $k\ge 0$
    \begin{equation}
    \label{lequ}
    P_{0}(\tilde{T}_{-L}^{l}<T_{L}^{l})\leq \frac{1}{c_{19}}\exp\left\{-\frac{c_{19}S_1^k}{v^{k}}\right\}
    \end{equation}
    Now, the  computation 
    $S_1^k=m_k L_k=(m_k+1)L_k-L_k\geq L - \frac{u_{0}}{2 \alpha c_{1}}v^{-k}L,$
    replaced at (\ref{lequ}), gives us
    $$P_{0}(\tilde{T}_{-L}^{l}<T_{L}^{l})\leq
\frac{1}{c_{19}} \exp\left\{-\frac{c_{19}L\left(1-\frac{u_{0}}{2 \alpha c_{1}}v^{-k}\right)}{v^{k}}\right\}$$
    So that, there exists $c_{20}$ such that
    $$P_{0}(\tilde{T}_{-L}^{l}<T_{L}^{l})\leq
\frac{1}{c_{20}} \exp\left\{-c_{20}\frac{L}{v^{k}}\right\}$$
Using now (\ref{equ3}) we conclude that there is a constant
$c_{16}$ such that for $L\ge L_0$ one has that

    $$P_{0}(\tilde{T}_{-L}^{l}<T_{L}^{l}) \leq
\frac{1}{c_{16}} \exp\left\{-c_{16}L\exp\left\{-\frac{1}{c_{16}}\log_8^{\left(\left[\frac{k+1}{2}\right]\right)}(L) \right\}\right\}.$$
Choosing  $n(L)=\left[\frac{k+1}{2}\right]$ we conclude the proof.

    Now, we need to prove (\ref{desfinal}). Using Proposition \ref{sznitman}, with $\mathcal{\widetilde{B}}$ and $\mathcal{B}_k$ instead of $\mathcal{B}_1$ and $\mathcal{B}_0$, we have:\\

    $$\mathbb{E}[\widehat{\rho}_{k}^{a_{k+1}}]\leq c_3 \left\{\kappa^{-10c_1 S_1^k}\varphi_k^{\frac{m_k^2}{12}}+\sum\limits_{0\leq j\leq m_{k}+1}\varphi_k^{\frac{m_k+j-1}{2}}\right\}$$

    So that

    $$(\widetilde{S}_2^k)^{d-1} S_1^k \mathbb{E}[\widehat{\rho}_{k}^{a_{k+1}}]\leq c_3 (S_2^k)^{d-1} S_1^k \left\{\kappa^{-10c_1 S_1^k}\varphi_k^{\frac{m_k^2}{12}}+\sum\limits_{0\leq j\leq m_{k}+1}\varphi_k^{\frac{m_k+j-1}{2}}\right\}$$
    From (\ref{mk}) and Lemma \ref{escale} (consequently we can use Lemma \ref{induction}), the following inequalities hold:

    \begin{equation}
    \label{aux1}
\kappa^{-10c_1 S_1^k} \varphi_k^{\frac{m_k^2}{24}} \leq \kappa^{-10c_1 S_1^k} \kappa^{\frac{m_k S_1^k u_k}{24}}\leq 1.
\end{equation}

    Then, inequality (\ref{aux1}) and the fact that $m_k-1\geq \frac{m_k}{2}$, imply that

    $$(\widetilde{S}_2^k)^{d-1} S_1^k  \mathbb{E}[\widehat{\rho}_{k}^{a_{k+1}}] \leq c_3 (\widetilde{S}_2^k)^{d-1} S_1^k \left\{\varphi_k^{\frac{m_k^2}{24}}+ S_1^k \varphi_k^{\frac{m_k}{4}}\right\}.$$

    So that
    $$(\widetilde{S}_2^k)^{d-1} S_1^k  \mathbb{E}[\widehat{\rho}_{k}^{a_{k+1}}] \leq 2c_3 (\widetilde{S}_2^k)^{d-1} (S_1^k)^2 \varphi_k^{\frac{m_k}{8}}\kappa^{u_{k+1}S_1^k}.$$
    Where, it was used the result of Lemma \ref{induction}.
    Finally, note that to finish the proof we have to show that
    \begin{equation}
    \label{last}
    2c_3 (\tilde S^k_2)^{d-1} (S_1^k)^2\varphi_k^{\frac{m_k}{8}}\leq 1.
    \end{equation}
        By our definitions in (\ref{lkformk}),

    $$(\widetilde{S}_2^k)^{d-1} (S_1^k)^2=m_k^{6d-4}\widetilde{L}_k^{d-1}L_k^2.$$
    Therefore, by Lemma \ref{escale} and its consequence Lemma \ref{induction}, the left hand side of inequality (\ref{last}) is smaller than

    $$m_k^{6d-4}\tilde{L}_k^{d-1}L_k^2 \kappa^{u_{k+1}m_k L_k}.$$
    However, as $d$ is fixed, and $k$ is large, it is clear that

    $$\tilde{L}_k^{d-1}L_k^2 \kappa^{\frac{u_{k+1}m_k L_k}{2}} \leq 1$$
    and

    $$c_3 m_k^{6d-4} \kappa^{\frac{u_{k+1}m_k L_k}{2}} \leq 1.$$
    This completes the proof.

    \end{proof}

\medskip

    It is now easy to check that Proposition \ref{final-prop} 
implies  Theorem \ref{theorem1} with the function $\log x$ replaced
by $\log_8 x$. Indeed, note that (\ref{ce1})
is equivalent to the effective criterion.
On the other hand, using the fact that for every $x>0$,
$\log x\ge \log_8 x$, we can then obtain Theorem \ref{theorem1}.

\bigskip
\noindent{\bf Acknowledgments:}
We  thank  A.-S. Sznitman for suggesting that
the decay implied by condition $(T')$ could be improved.

% Every LaTeX document must end with \end{document}.

\bigskip
\bigskip
\bigskip
\bigskip
\end{document}